\documentclass[12pt,oneside,reqno]{amsart}

\newcommand{\supp}{\operatorname{supp}}

\usepackage{amssymb}
\usepackage{amsthm}
\usepackage{amsxtra}
\usepackage{graphicx}
\usepackage{color}
\usepackage{xcolor}
\usepackage{mathrsfs}

\newtheorem{theorem}{Theorem}

\newtheorem{lemma}[theorem]{Lemma}

\theoremstyle{remark}

\numberwithin{equation}{section}

\numberwithin{theorem}{section}

\numberwithin{table}{section}

\numberwithin{figure}{section}

\ifx\pdfoutput\undefined
  \DeclareGraphicsExtensions{.pstex, .eps}
\else
  \ifx\pdfoutput\relax
    \DeclareGraphicsExtensions{.pstex, .eps}
  \else
    \ifnum\pdfoutput>0
      \DeclareGraphicsExtensions{.pdf}
    \else
      \DeclareGraphicsExtensions{.pstex, .eps}
    \fi
  \fi
\fi

\title{Focusing solutions of the Vlasov-Poisson System}

\author{Katherine Zhiyuan Zhang}
\address{Brown University}

\begin{document}

\maketitle

\begin{abstract}

We study smooth, spherically-symmetric solutions to the Vlasov-Poisson system and relativistic Vlasov-Poisson system in the plasma physical case. We construct solutions that initially possess arbitrarily small $C^k$ norms ($k \geq 1$) for the charge densities and the electric fields, but attain arbitrarily large $L^\infty$ norms of them at some later time.

\end{abstract}

\section{Introduction}

We consider the one-species classical Vlasov-Poisson System (VP):
\begin{equation}
\partial_t f + v \cdot \nabla_x f +  E   \cdot \nabla_v f =0 \ ,
\end{equation}
\begin{equation}
E(t, x) = \int_{\mathbb{R}^3} \frac{x-y}{|x-y|^3} \rho (t, y) dy \ . 
\end{equation}
Here, $f (t, x, v) \geq 0$ is the density distribution of the particles. In the equation, $x \in \Omega \subset \mathbb{R}^3$ is the particle position, $v \in \mathbb{R}^3$ is the particle momentum, and $E$ is the electric field. Moreover, the charge density $\rho$ is defined as
\begin{equation}
\rho = \int_{\mathbb{R}^3} f (t, x, v) dv   \ .
\end{equation}

We also consider the Vlasov-Poisson System in the relativistic setting (RVP):
\begin{equation}
\partial_t f + \hat{v} \cdot \nabla_x f +E  \cdot \nabla_v f =0 \ ,
\end{equation}
\begin{equation}
E(t, x) = \int_{\mathbb{R}^3} \frac{x-y}{|x-y|^3} \rho (t, y) dy \ . 
\end{equation}
Here $\hat{v} = \frac{v}{\sqrt{1 +|v|^2}}$ is the velocity.

The systems VP and RVP enjoy the conservation of the total mass
\begin{equation}
M (t) = \int_\Omega \int_{\mathbb{R}^3}  f (t, x, v) dv dx =  \int_\Omega \int_{\mathbb{R}^3}  f_0 ( x, v) dv dx : = M \ , \forall t  \ .
\end{equation}
Here $f_0$ is the initial particle density. 

We assume spherical symmetry in the problem. It is known that spherically symmetric initial data give rise to global-in-time, spherically symmetric solutions to the two systems, see \cite{H1}, \cite{LP1}, \cite{P1}. Also, \cite{H1} tells us that the solutions must be finite in the $L^\infty $ sense.

The behavior of the solution to VP has been an important topic that caught wide attention. In the paper \cite{BCP2} by J. Ben-Artzi, S. Calogero and S. Pankavich, it is shown that one can construct solutions of VP such that the particle density and the electric field are initially as small as desired, but become large as desired at some later time, as stated in the following theorem: \\
\textbf{Theorem} (J. Ben-Artzi, S. Calogero and S. Pankavich) For any positive constants $\eta$, $N$, there exists a smooth, spherically symmetric solution of VP, such that  
\begin{equation}  
\| \rho (0) \|_{L^\infty} \  , \ \| E (0) \|_{L^\infty} \leq \eta \ , 
\end{equation}
while for some $T >0$, 
\begin{equation} 
\| \rho (T) \|_{L^\infty} \  , \ \| E (T) \|_{L^\infty} \geq N \ , 
\end{equation}
An analogous result for RVP is proved in \cite{BCP1} by the same authors. Their results in \cite{BCP2} and \cite{BCP1} were inspired by a similar study by G. Rein and L. Teagert \cite{RT1}, which proves the existence of focusing solutions for the gravitational Vlasov-Poisson system, in which the electric force provides an attractive effect instead of a repulsive effect on the particles. It is striking that in contrast to the results in \cite{BCP2} and \cite{BCP1}, a classical estimate given by E. Horst in \cite{H1} in 1990 shows that any spherically symmetric solution must decay for $t$ sufficiently large. Namely, there exists $C>0$ and $T>0$, such that 
$$ \| \rho (t) \|_{L^\infty} \leq C t^{-2} \ , \ \| E(t) \|_{L^\infty} \leq C t^{-2}   $$
for all $t \geq T$. There is no contradiction between the conclusions in \cite{BCP2}, \cite{BCP1} and \cite{H1}.  

However, in the examples provided in \cite{BCP2} and \cite{BCP1}, the initial data are actually large in a $C^1$ sense. Indeed, every datum $\rho (0)$ constructed in \cite{BCP2} and \cite{BCP1} is supported on a spherical shell of radius $a_0$ (which is selected to be a large positive number) and the thickness of the shell is $\sim \epsilon^3$, where $\epsilon \lesssim a_0^{-1} N^{-1}$ is required to be sufficiently small. The $L^\infty$ norm of $\rho (0)$ is of size $\sim a_0^{-3}$. Hence the $C^1_x$ norm of $\rho (0)$ is of size $\sim a_0^{-3} \epsilon^{-3} \gtrsim N^3$.

In this paper, we consider initial data supported on an arbitrary shell that has small $C^k$ norm ($k \geq 1$), and obtain solutions that become large and are concentrated near the origin at some later time. We call them "focusing solutions". Specifically, we prove, for VP:
\begin{theorem} \label{mainresultVP}
For any positive integer $k$ and positive constants $\eta$, $N$, $b$, $\epsilon_0$, there exists a smooth, spherically symmetric solution of VP, such that 
\begin{equation}
\supp_x f(0, x, v) = \{  \frac{1}{2} b \leq |x| \leq \frac{3}{2} b   \} 
\end{equation}
and 
\begin{equation}  
\| \rho (0) \|_{C^k} \  , \ \| E (0) \|_{C^k} \leq \eta \ , 
\end{equation}
while there exists a function $T = T(b) $, which is increasing for $b >1$, such that
\begin{equation} 
\| \rho (T) \|_{L^\infty_{|x|\leq  \epsilon_0}} \  , \ \| E (T) \|_{L^\infty } \geq N \ , 
\end{equation}
\end{theorem} 

\begin{flushleft} \textbf{Remark 1.2} 
\it{i) The data constructed in Theorem \ref{mainresultVP} do not belong to the class $\mathfrak{J}$ defined in \cite{BCP2}. In Theorem \ref{mainresultVP}, the constant $b$ can be chosen arbitrarily. This means that $\supp_x f (0)$ can be a thick spherical shell with both its central radius and thickness $\sim b$. On the other hand, the data in the class $\mathfrak{J}$ have particle density functions supported in a thin shell with thickness $\sim \epsilon^3$, where $\epsilon>0$ is chosen to be a small number. } \\
\textit{ii) In case $1 \leq k \leq 5$ we actually construct a function $T(b)$ that is increasing on $(0, +\infty)$. } \\
\textit{iii) We also construct focusing solutions for RVP with the $C^k$ norms for $\rho (t, x)$ and $E(t, x)$ being small at time zero and growing large at some later time, see Theorem \ref{mainresultRVP}.} \\
\textit{iv) Theorem \ref{mainresultVP} actually enables us to easily give a similar result in the setting with a bounded domain, see Remark 3.1 in Section 3. }\\
\textit{v) By perturbation, results analogous to Theorem \ref{mainresultVP} and Theorem \ref{mainresultRVP} hold for multi-species VP and RVP (with one of the species dominating the behaviour of the plasma, so the plasma is "almost single species"). }
\end{flushleft}


The contents of the paper are arranged as follows. In Section 2, we give some lemmas that describe the particle trajectories, which allow us to observe the focusing phenomena. Section 3 is devoted to the proof of Theorem \ref{mainresultVP}, which involves a careful selection of parameters and computation of the norms of $\rho (t, x)$ and $E(t, x)$. At the end of Section 3 we also give the corollary on the setting with a bounded domain. The analogous result on RVP to Theorem \ref{mainresultVP} will be stated (see Theorem \ref{mainresultRVP}) and proved in Section 4.

\section{Characteristics and Useful Lemmas}

A spherically symmetric solution to VP or RVP can be described as 
$$(f(t, r, w, l), E(t, r)) $$
where the spatial radius $r$, radial velocity $w$ and square of the angular momentum $l$ are defined as follows:
\begin{equation}
r = |x| \ , \ w = \frac{x \cdot v}{r} \ , \ l = |x \times v |^2 \ , 
\end{equation}
By a change of variables $(x, v) \rightarrow (r, w, l)$, using $ \nabla_x r = \frac{1}{r} x $, $ \partial_{x_j} w = v_j (\frac{1}{r} - \frac{x_j^2}{r^3} ) $, $ \nabla_v w = \frac{1}{r} x $, $ \nabla_x l = 2 |v|^2 x - 2 rw v $, $\nabla_v l = 2 r^2 v - 2rw x$, we reduce the Vlasov equation to
\begin{equation}  \label{VPreformulated}
\partial_t f +  w \partial_r f + \big( \frac{l}{r^3  } + \frac{m(t, r)}{r^2} \big) \partial_w f =0 
\end{equation}
for the system VP, and similarly, for RVP the Vlasov equation is reduced to
\begin{equation}
\partial_t f + \frac{w}{\sqrt{1+ w^2 + l r^{-2}}} \partial_r f + \big( \frac{l}{r^3 \sqrt{1 +  w^2 + l r^{-2} } } + \frac{m(t, r)}{r^2} \big) \partial_w f =0 \ .
\end{equation}
 
Here
\begin{equation} \label{mtr}
m(t, r) = 4 \pi \int^r_0 s^2 \rho (t, s) ds \ ,
\end{equation}
and
\begin{equation} \label{rhotr}
\rho (t, r) = \frac{\pi}{r^2} \int^\infty_0 \int^\infty_{-\infty} f(t, r, w, l) dw dl \ .
\end{equation}
The electric field is then 
\begin{equation} \label{electricfieldexpression}
E(t,x) = \frac{m(t, r) x}{r^3} \ ,
\end{equation}
since we can verify  
\begin{equation}
\begin{split}
div \big(  \frac{m(t, r) x}{r^3}  \big)
& = 4 \pi \big( \frac{r^2 \rho (t, r)}{r^3} -3 \frac{ \int^r_0 s^2 \rho (t, s) ds }{r^4 } \big) e_r \cdot x \\
& + 4 \pi  \frac{\int^r_0 s^2 \rho (t, s) ds}{r^3} \cdot 3 \\
& = 4 \pi \rho (t, r) \ , \\
\end{split}
\end{equation}
and
$$ curl \big(  \frac{m(t, r) x}{r^3}  \big)= 4 \pi \nabla \big( \frac{ \int^r_0 s^2 \rho (t, s) ds  }{r^3} \big) \times x + 4 \pi \frac{ \int^r_0 s^2 \rho (t, s) ds  }{r^3} \ (curl \ x ) = 0 \ .  $$
This is enough to verify that the formula $E(t,x) = \frac{m(t, r) x}{r^3} $ gives an $E$ that satisfies the Vlasov-Maxwell system with $B =0$. To see it matches the expression $E(t, x) = \int_{\mathbb{R}^3} \frac{x-y}{|x-y|^3} \rho (t, y) dy$, note that they could only differ by the gradient of a harmonic function $g(x)$. However, since we assume $E(t, x)$ has finite $L^\infty$ norm, $g(x)$ must be finite too. By Liouville's Theorem, $g$ must be a constant, which implies that the two $E(t, x)$ expressions matches each other.
 
The total mass of the plasma is 
\begin{equation}
M = 4 \pi^2 \int^\infty_0 \int^\infty_{-\infty}  \int^\infty_0 f_0( r, w, l) dl dw dr \ . 
\end{equation}

Next we introduce the characteristics for VP and RVP, as well as the lemmas that give detailed information for the particle trajectories.

The forward characteristics $(R(s), W(s),  L(s))$ of the Vlasov equation in the non-relativistic setting are described by the following ODE system:
\begin{equation} \label{particletrajectoryODEVP}
\begin{split}
& \frac{d R}{ds} = W \ , \\
& \frac{d W}{ds} =  \frac{L}{R^3 } + \frac{m(s, R)}{R^2} \ ,  \\
& \frac{dL}{ds} =0 \ . \\
\end{split}
\end{equation}
for $s \geq 0$, with the initial conditions 
\begin{equation}  \label{particletrajectoryinitialcondition}
R(0) =r \ , \ W(0) =w \ , \ L(0) =l \ . 
\end{equation}

We have the following lemma from \cite{BCP2}:

\begin{lemma} \label{behaviorofcharacteristicsVP}
Let $r>0$, $l >0$, $w < 0$ be given, and let $(R(t), W(t), L(t))$ be a solution to \eqref{particletrajectoryODEVP} and \eqref{particletrajectoryinitialcondition} for all $t \geq 0$.
Then we have:   \\
(1) There exists a unique $T_0 > 0$ such that $W(t) < 0$ for $t \in [0, T_0 )$, $W(T_0) =0 $, and $W(t) > 0$ for $t \in (T_0, +\infty )$.   \\
(2) $T_0$ satisfies   
$$ T_0 \geq \frac{r}{|w|} \big(  1 - \sqrt{\frac{l +Mr}{r^2  w^2 + l + Mr  }}  \big)  \ .   $$ 
(3) For all $t \in [0, T_0)$, we have
$$ R(t)^2 \leq  ( r + w t   )^2 + (l r^{-2} + M r^{-1} ) t^2   \ . $$
\end{lemma}

\begin{proof}
Please see Lemma 3 and 4 in \cite{BCP2} for details. 
\end{proof}

For RVP the forward characteristics $(R(s), W(s),  L(s))$ of the Vlasov equation are described by
\begin{equation} \label{particletrajectoryODE}
\begin{split}
& \frac{d R}{ds} = \frac{W}{\sqrt{1+W^2 +L R^{-2}}} \ , \\
& \frac{d W}{ds} =  \frac{L}{R^3 \sqrt{1 +  W^2 + L R^{-2} } } + \frac{m(s, R)}{R^2} \ ,  \\
& \frac{dL}{ds} =0 \ . \\
\end{split}
\end{equation}
for $s \geq 0$, with the initial conditions \eqref{particletrajectoryinitialcondition}.

We introduce the following lemma from \cite{BCP1}:
\begin{lemma} \label{behaviorofcharacteristics}
Let $r>0$, $l >0$, $w < 0$ be given, and let $(R(t), W(t), L(t))$ be a solution to \eqref{particletrajectoryODE} and \eqref{particletrajectoryinitialcondition} for all $t \geq 0$, and define
$$ D = l + M r \sqrt{1+ w^2 + l r^{-2}} \ .$$ 
Then we have:  \\
(1) There exists a unique $T_0 > 0$ such that $W(t) < 0$ for $t \in [0, T_0 )$, $W(T_0) =0 $, and $W(t) > 0$ for $t \in (T_0, +\infty )$. \\
(2) $T_0$ satisfies 
$$ T_0 \geq r \big( 1 - \sqrt{\frac{D}{r^2 w^2 + D}}   \big)  \ .   $$
(3) There holds
$$ R_- \leq R(T_0) \leq R_+ \ , $$
where
$$ R_- : =  r \sqrt{\frac{l}{r^2 w^2 +l}}  \ , \ R_+ : =  r \sqrt{\frac{D}{r^2 w^2 +D}} \ .  $$
(4) For all $t \in [0, T_0)$, we have
$$ W(t)^2 + l R(t)^{-2} \leq w^2 + l r^{-2}  \ . $$
(5) For all $t \in [0, T_0)$, we have
$$ R(t)^2 \leq \big( r - \frac{|w|}{ \sqrt{1+ w^2 + l r^{-2}}} t  \big)^2 + \frac{D}{r^2 (1+ w^2 + l r^{-2} ) } t^2   \ . $$
\end{lemma}

\begin{proof}
Please see Lemma 3 in \cite{BCP1} for details. 
\end{proof}

We denote
\begin{equation}
S(t) : = \{  (r,w, l) : f(t, r, w, l) > 0  \} \ . 
\end{equation}
In particular, 
\begin{equation}
S(0)  = \{  (r,w, l) :   f_0( r, w, l) > 0  \} \ .
\end{equation}
Also, we denote
\begin{equation}
S_+  = \{  (r,w, l) :  r >0, l> 0, w <0, f_0( r, w, l) > 0  \} \ .
\end{equation} 
We choose $T_0 > 0$ as in Lemma \ref{behaviorofcharacteristicsVP} and \ref{behaviorofcharacteristics} in the VP and RVP settings, respectively. We present the following lemma given in \cite{BCP2} and \cite{BCP1} in order to describe the concentrating phenomenon:
\begin{lemma} \label{concentrationlemma}
Let $f (t, r, w, l)$ be a spherically-symmetric solution of RVP or VP with associated charge density $\rho (t, r)$ and electric field $E(t, x)$. Let 
$$(R(t, r, w, l), W(t, r, w, l), L(t, r, w, l) )$$
be a solution to the equations of the particle trajectories with the initial condition \eqref{particletrajectoryinitialcondition}. If at some time $T >0$ we have
$$ \sup_{(r, w, l) \in S(0)} R (T, r, w, l) \leq K  \ , $$
then
$$ \| \rho (T) \|_{L^\infty_{|x|\leq K }} \geq \frac{3 M}{4 \pi K^3} \ , \  \| E (T) \|_{L^\infty} \geq \frac{  M}{ K^2}  \ .     $$
\end{lemma}

\begin{proof}
Please see Lemma 5 in \cite{BCP2} or Lemma 4 in \cite{BCP1} for details. 
\end{proof}

\section{Focusing Solutions to the Nonrelativistic VP System}

Now we are ready to establish Theorem \ref{mainresultVP}.

\begin{proof}
Without the loss of generality, we assume $0 <\eta< 1$ and $N >1$ in the proof. We set up two constants $a_0 >1$ and $0 < \epsilon <1$ to be determined.

Let $H : [0, \infty) \rightarrow [0, \infty)$ be any function satisfying $ \int_{\mathbb{R}^3} H(|u|^2) du = \frac{3}{4 \pi} $ with $\supp (H) \subset [0, 1]$, and rescale it for any $\epsilon \in (0, 1)$:
$$ H_\epsilon (|u|^2) = \frac{1}{\epsilon^3} H \big( \frac{|u|^2}{\epsilon^2} \big)   \ . $$
so that $ \int_{\mathbb{R}^3} H_\epsilon(|u|^2) du = \frac{3}{4 \pi} $ and $\supp (H_\epsilon) \subset [0, \epsilon^2]$. For any $\epsilon >0$, $x$, $v \in \mathbb{R}^3$, $a_0 >0$, define 
$$ h_{\epsilon, a_0} (x, v) = H_\epsilon \big(  | \frac{x}{\epsilon^2} + a_0 v|^2 \big) \ . $$
Then $ \int_{\mathbb{R}^3} h_{\epsilon, a_0} (x, v) dv = \frac{3}{4 \pi a_0^3} $ for all $x$. In the proof, we will choose $a_0$ sufficiently large and $\epsilon$ sufficiently small.  

Moreover, we choose a cut-off function $\chi_{0, 1} \in C^\infty ((0, \infty); [0, 1] ) $ that satisfies 
\begin{equation*}
\begin{split}
& |\chi_{0, 1} (x)| \leq 1 \ , \\
& \chi_{0, 1} (x) = 1 , \ \text{for} \ |x| <  \frac{1}{2} \ , \\
& \chi_{0, 1} (x) =0, \ \text{for} \ |x| > 1 \ . \\
\end{split}
\end{equation*}
Let $\chi_{m, n} (x) : = \chi_{0, 1} (\frac{1}{n} (x-m))$, and take $\chi (r) = \chi_{b, \frac{1}{2} b} (r)$, then $\chi (r)$ is a smooth function supported on $[\frac{1}{2} b, \frac{3}{2} b]$, and $\chi (r ) = 1 $ for $r \in [\frac{3}{4} b, \frac{5}{4}b ]$. We denote $\| \frac{d^k}{dr^k} \chi_{0, 1} (r) \|_{L^\infty} = \alpha_k $, where the $\alpha_k$'s are constants independent of $b$. In particular, $\alpha_0 = 1$. Then
\begin{equation} \label{chiVP}
\| \frac{d^k}{dr^k} \chi  (r) \|_{L^\infty} \leq c_k :=  (\frac{2}{b})^k \alpha_k   \ . 
\end{equation}
In particular $c_0 =1$. Define $c_{-1} = b/2$. We denote $d_k =  \sum^{k}_{j=-1} c_j$.

We choose the initial data to be 
\begin{equation} \label{f0VP}
f_0 (x, v) = h_{\epsilon, a_0 } (x, v) \chi (|x|) \Gamma (\frac{x \cdot v}{|x|}) \ . 
\end{equation}
where $\Gamma (w) = 0$ for $w \geq 0$, and $\Gamma (w) =1$ for $w <0$. Hence $S(0) \setminus S_+ $ is a zero-measure set. Note that $\supp_{x, v} f_0 = \{ (x, v) : |\frac{x}{\epsilon^2} + a_0 v  |^2  < \epsilon^2 \}$. 
We have 
$$0 \leq   \rho (0, r) = \int_{\mathbb{R}^3} h_{\epsilon, a_0 } (x, v) \chi (|x|) \Gamma (\frac{x \cdot v}{|x|}) dv =  \frac{3}{8 \pi a_0^3} \chi (r) \ .   $$
Also, we compute
\begin{equation} \label{MsizeVP}
\frac{3}{8} a_0^{-3} b^3  \leq M \leq 4 a_0^{-3} b^3 \ . 
\end{equation}

From $ |\frac{x}{\epsilon^2} + a_0 v  |^2  < \epsilon^2 $, using the angular coordinates and the identity $ |v|^2 = w^2 + lr^{-2} $, we obtain
\begin{equation}
\big( \frac{r}{\epsilon^2} + a_0 w \big)^2 + l \big( \frac{a_0}{r} \big)^2 < \epsilon^2  \ .
\end{equation}  
We obtain $ | \frac{r}{\epsilon^2} + a_0 w |^2 < \epsilon^2 $, which implies $ | r+ a_0 w \epsilon^2  | < \epsilon^3 $. Therefore,
$$ \big| \frac{r}{|w|} - a_0 \epsilon^2  \big| < \frac{\epsilon^3}{|w|}  $$ for $(r, w, l) \in S(0)$, since $\Gamma (w) $ is non-zero if and only if $w < 0 $. 

Also, from the definition of $\chi (r) $ we obtain, for any $(r, w, l) \in S(0)$:
\begin{equation} \label{rbound}
\frac{1}{2} b <   r <   \frac{3}{2}  b \ .
\end{equation}
We pick 
\begin{equation} \label{a0condition4}
a_0 > b 
\end{equation} 
and $\epsilon$ small enough, such that 
\begin{equation} \label{epsiloncondition1}
-\frac{1}{2} a_0^{-1} b  \epsilon^{-2} + \frac{\epsilon}{a_0} < -1, \ - \frac{3}{2} \epsilon^{-2} < - \frac{3}{2} a_0^{-1} b \epsilon^{-2} - \frac{ \epsilon}{a_0} , \  \big( \frac{3b/2}{a_0} \big)^2 \epsilon^2<1 . 
\end{equation}
Therefore
\begin{equation} \label{wbound}
- \frac{3}{2} \epsilon^{-2} < - \frac{3}{2} a_0^{-1} b \epsilon^{-2} - \frac{ \epsilon}{a_0} < w <  -\frac{1}{2} a_0^{-1} b  \epsilon^{-2} + \frac{\epsilon}{a_0}  \ , 
\end{equation}
\begin{equation} \label{lbound}
l < \big( \frac{r}{a_0} \big)^2 \epsilon^2  <  \big( \frac{3b/2}{a_0} \big)^2 \epsilon^2<1 \ .
\end{equation}
We see that the cut-off $\Gamma$ does not have any effect on the smoothness of $f_0$ since $|w| > |-\frac{1}{2} a_0^{-1} b  \epsilon^{-2} + \frac{\epsilon}{a_0} | > 1$. 
We have
$$ \big| \frac{r}{|w|} - a_0 \epsilon^2  \big| < \frac{\epsilon^3}{|w|} < \epsilon^3  $$
since $|w| > 1$. Hence 
\begin{equation}  \label{r/wbound}
 \frac{r}{|w|} \geq  a_0 \epsilon^2 - \epsilon^3 \ . 
\end{equation}

By \eqref{electricfieldexpression} and \eqref{MsizeVP}, we can take
\begin{equation} \label{a0condition3}
a_0 \geq \max \{\eta^{-1/3}  , 4 \eta^{-1/3} b^{1/3} \} \ , 
\end{equation}
so that
\begin{equation}  \label{rho0bound}
 \| \rho (0) \|_{L^\infty} \leq \frac{3}{8 \pi a_0^3}  \leq \eta \ , 
\end{equation} 
\begin{equation}
 \| E (0) \|_{L^\infty} \leq \frac{4 M}{b^2} \leq  \frac{ 16 a_0^{-3} b^3}{b^2} \leq  \frac{1}{2} \eta  \ .
\end{equation}
Noticing that $\rho ( 0)$ and $E ( 0)$ are only functions of $r$, we have, due to \eqref{f0VP} and \eqref{chiVP}:
\begin{equation} \label{rho0Ckbound}
\| \rho (0) \|_{C^k_x}  \leq \sum_k  \| \frac{\partial^k}{\partial r^k}  ( \frac{3}{8 \pi a_0^3} \chi (r)) \|_{L^\infty}   \leq     \frac{3  }{8 \pi a_0^3 }   \sum^{k }_{j=0} c_j  \leq \eta \ ,
\end{equation} 
if we take
\begin{equation} \label{a0condition1}
a_0 \geq \frac{8 \pi}{3}  d_k^{1/3}  \eta^{-1/3} \ .
\end{equation}

Recall that in Theorem \ref{mainresultVP}, we require 
$$ \| E (0) \|_{C^k_x}  \leq   \eta  \ .  $$
We will show, for any $k \geq 0$, 
\begin{equation} \label{E0Ckboundcomputation}
\| E(0) \|_{C^k} \leq C_0 d_{k-1} a_0^{-3} 
\end{equation}
for some constant $C_0 >1$ to be chosen below, which only depends on $k$. If $k =0$, we can take $C_0 \geq 32$. In case $k=1$, using \eqref{mtr}, \eqref{rhotr}, \eqref{electricfieldexpression}, \eqref{MsizeVP}, \eqref{rho0bound}, and that $\frac{1}{2} b  < r < \frac{3}{2} b $, we compute
\begin{equation}
\begin{split}
| \partial_{x_j} E_i (0, r) |
& = | \partial_{x_j}  ( \frac{m (0, r)}{r^3} x_i )   |  \\
& \leq  |  \frac{\partial_{x_j} m (0, r) }{r^3} x_i  | + | m (0, r) \frac{-3 x_j x_i }{r^5} | + | \frac{m(0, r) }{r^3} \delta_{ij} |  \\
& \leq 4 \pi  | r^2 \rho (0, r) \frac{x_j}{r} \frac{x_i }{r^3}   | + | m (0, r) \frac{-3 x_j x_i }{r^5} | + | \frac{m(0, r) }{r^3} \delta_{ij} |  \\
& \leq 4 \pi \big[  | r^2 \frac{3}{ 8 \pi a_0^3}  \frac{x_j}{r} \frac{x_i }{r^3}   | + | \frac{3}{8 \pi a_0^3} \frac{1}{3} r^3 \frac{-3 x_j x_i }{r^5} | +  | \frac{3}{8 \pi a_0^3} \frac{1}{3} r^3 \frac{1}{r^3} \delta_{ij} |  \big] \\
& \leq   10   c_0 a_0^{-3}  \ , \\
\end{split}
\end{equation} 
so we can take $C_0  \geq \max \{ 32, 10 \} = 32 $ for $k =1$.   
Similarly, in case $k=2$, we make use of \eqref{mtr}, \eqref{rhotr}, \eqref{electricfieldexpression} to compute 
\begin{equation}
\begin{split}
& \quad | \partial^2_{x_j x_l} E_i (0, r) | \\
& =  | \partial_{x_l}  (  \frac{\partial_{x_j} m (0, r) }{r^3} x_i  +  m (0, r) \frac{-3 x_j x_i }{r^5}  + \frac{m(0, r) }{r^3} \delta_{ij})   |  \\
& \leq  |  \frac{\partial_{x_j x_l} m (0, r) }{r^3} x_i  | + | \partial_{x_j} m (0, r) \frac{-3 x_l x_i }{r^5} | + | \frac{\partial_{x_j} m(0, r) }{r^3} \delta_{il} |  \\
& \quad + |  \frac{\partial_{x_l} m (0, r) }{r^3} x_i x_j  (-3)| + | m (0, r) \frac{(-3) (-5) x_l x_i x_j }{r^7} | + | \frac{m(0, r) }{r^5} (-3) (x_i \delta_{jl} + x_j \delta_{il} ) |  \\
& \quad + |  \frac{\partial_{ x_l} m (0, r) }{r^3} \delta_{ij}  | + | m (0, r) \frac{-3 x_l  }{r^5} \delta_{ij} | + | \frac{m(0, r) }{r^3} \cdot 0 |  \\
& \leq 4 \pi \big[ 4 \pi |  \frac{\partial_{x_j } (  r^2  \rho (0, r) \frac{x_l}{r}) }{r^3} x_i  | + | r^2 \rho (0, r) \frac{x_j}{r} \frac{-3 x_l x_i }{r^5} | + | r^2 \rho (0, r) \frac{ x_j }{r} \frac{1}{r^3} \delta_{il} | \big] \\
& \quad + 4 \pi | r^2 \rho(0, r) \frac{x_l}{r} \frac{-3 x_i x_j}{r^5}  | + | m(0, r)  \frac{(-3) (-5) x_l x_i x_j }{r^7} | \\
& \quad + | m(0, r) \frac{1 }{r^5} (-3) (x_i \delta_{jl} + x_j \delta_{il} ) |  + 4 \pi | r^2 \rho (0, r) \frac{x_l}{r}  \frac{1}{r^3} \delta_{ij}  | +  | m(0, r) \frac{-3 x_l  }{r^5} \delta_{ij} |    \\
& \leq 4 \pi \big\{ 4 \pi  \frac{x_i}{r^3} \big[ | \rho (0, r) \frac{x_j x_l}{r} | +| r \frac{\partial}{\partial r} \rho (0, r) \frac{x_j}{r} x_l | + | r \rho (0, r) \delta_{jl} | \big] \\
& \quad + | r^2 \rho (0, r) \frac{x_j}{r} \frac{-3 x_l x_i }{r^5} | + | r^2 \rho (0, r) \frac{ x_j }{r} \frac{1}{r^3} \delta_{il} |  \\
& \quad + | r^2 \rho(0, r) \frac{x_l}{r} \frac{-3 x_i x_j}{r^5}  | + \frac{1}{4 \pi} | m(0, r)  \frac{(-3) (-5) x_l x_i x_j }{r^7} |  \\
& \quad + \frac{1}{4 \pi} | m(0, r) \frac{1 }{r^5} (-3) (x_i \delta_{jl} + x_j \delta_{il} ) |   + | r^2 \rho (0, r) \frac{x_l}{r}  \frac{1}{r^3} \delta_{ij}  | + \frac{1}{4 \pi} | m(0, r) \frac{-3 x_l  }{r^5} \delta_{ij} |   \big\} \ , \\
\end{split}
\end{equation} 
Using that $m (0, r) \leq M$, \eqref{MsizeVP}, \eqref{rho0bound}, \eqref{rho0Ckbound} and $\frac{1}{2} b  < r < \frac{3}{2} b $, we obtain
\begin{equation}
\begin{split}
& \quad | \partial^2_{x_j x_l} E_i (0, r) |  \\
& \leq 4 \pi \big\{ 4 \pi  \frac{x_i}{r^3} \big[ | \frac{3}{8\pi a_0^3} \frac{x_j x_l}{r} | +| r  \frac{3}{8 \pi a_0^3} c_1  \frac{x_j}{r} x_l | + | r \frac{3}{8 \pi a_0^3} \delta_{jl} | \big] \\
& \quad + | r^2 \frac{3}{8 \pi a_0^3} \frac{x_j}{r} \frac{-3 x_l x_i }{r^5} | + | r^2 \frac{3}{8 \pi a_0^3} \frac{ x_j }{r} \frac{1}{r^3} \delta_{il} |  \\
& \quad + | r^2 \frac{3}{8 \pi a_0^3} \frac{x_l}{r} \frac{-3 x_i x_j}{r^5}  | + | \frac{1}{8 \pi} \frac{r^3}{a_0^3}   \frac{(-3) (-5) x_l x_i x_j }{r^7} | +  | \frac{1}{8 \pi} \frac{r^3}{a_0^3}  \frac{1 }{r^5} (-3) (x_i \delta_{jl} + x_j \delta_{il} ) |  \\
& \quad + | r^2 \frac{3}{8 \pi a_0^3} \frac{x_l}{r}  \frac{1}{r^3} \delta_{ij}  | +  | \frac{1}{8 \pi} \frac{r^3}{a_0^3}  \frac{-3 x_l  }{r^5} \delta_{ij} |   \big\} \\
& \leq   100  c_1  a_0^{-3}   \ , \\
\end{split}
\end{equation} 
Hence for $k=2$ we can take $C_0  \geq \max \{ 32, 10,  100 \} =100 $.   

For larger $k$'s the inequality \eqref{E0Ckboundcomputation} can be deduced similarly with properly chosen $C_0$. 


We will choose 
\begin{equation} \label{a0condition2}
a_0 \geq 2 C_0^{1/3} \eta^{-1/3} d_{k-1} \ ,
\end{equation}
so that
$$ \| E (0) \|_{C^k_x}  \leq C_0 d_{k-1} a_0^{-3} \leq   \eta  \ . $$

Now we can set up the constants $a_0$ and $\epsilon$. Combining \eqref{a0condition4}, \eqref{a0condition3}, \eqref{a0condition1} and \eqref{a0condition2}, we take $a_0$ to be such that
\begin{equation} \label{a0constraint}
a_0 \geq  20 d_{k} \eta^{-1/3}  C_0^{1/3}  \ .
\end{equation}
Therefore $ a_0 \geq  2 C_0^{1/3} \eta^{-1/3} d_{k-1}  > 1  $, since $d_k > d_{k-1} \geq 1$, $\eta < 1$, and $C_0$ is chosen to be greater than $1$. 
Moreover, we take
\begin{equation} \label{epsilonconstraint}
\begin{split}
& \epsilon \leq \beta_0 : = \min \{ \frac{1}{60} a_0^{-1} b, \frac{1}{60} a_0^{-1/3} b^{1/3} ,\frac{1}{60} (b^{-2} + 4 a_0^{-3} b^2 )^{-1/2} , \\
& \qquad \frac{1}{100} a_0^{-1} b N^{-1/3}, \frac{1}{60} a_0^{-3/2} b^{3/2} N^{-1/2} , \frac{\epsilon_0}{4} \}  \ , \\
\end{split}
\end{equation}
so \eqref{epsiloncondition1} is satisfied.

Next we construct the function $T = T(b)$. Recall that $a_0$ and $\epsilon$ depend on $b$. We define
\begin{equation}  \label{Tdefinition}
T = T(b) := a_0 \epsilon^2  - \epsilon^3  - \frac{100 a_0^2}{b^2} \epsilon^4   
\end{equation}
for $b \in (0, +\infty)$. We have $T >0$ due to the constraints above on $a_0$ and $b$.

We now prove that $T(b)$ can be constructed as an increasing function for $b >1$. Indeed, for $b> 1$, we want to take $a_0$ and $\epsilon$ such that \eqref{a0constraint} and \eqref{epsilonconstraint} are satisfied.
From the definition of $d_k$, we notice that there exists $n$ being large enough and independent of $b$, such that 
\begin{equation} \label{ndefinition}
nb  \geq  20  d_{k } \eta^{-1/3}  C_0^{1/3}   
\end{equation}
holds for all $b > 1$. Hence we can take $a_0 = nb$ so that \eqref{a0constraint} holds. Also, from \eqref{a0constraint} and the definition \eqref{epsilonconstraint} of $\beta_0$, we observe that for $b >1$ there exists a constant $C$ independent of $b$, such that
\begin{equation} \label{Cdefinition}
C \leq \beta_0  \ .
\end{equation}
Therefore we can take $\epsilon = C  $, so that \eqref{epsilonconstraint} is satisfied. Using these chosen values of $a_0$ and $\epsilon$, we have, for $b >1$,
\begin{equation} \label{Tdefinitionbgeq1}
T = T (b) = a_0 \epsilon^2  - \epsilon^3  - \frac{100 a_0^2}{b^2} \epsilon^4 = n C^2  b  - C^3   - 100 C^4 n^2 \ ,
\end{equation} 
which is an increasing function of $b \in (1, +\infty)$.

We choose $T_0 = T_0 (r, w, l) > 0$ for $(r, w, l) \in S_+$ as in Lemma \ref{behaviorofcharacteristics}. From the lemma we have
\begin{equation}
\frac{d R}{dt} \leq 0 \ , \ t \in [0, T_0] \ ,
\end{equation}
and 
\begin{equation}
T_0 \geq \frac{r}{|w|} \big(  1- \sqrt{\frac{l + Mr}{ r^2 w^2 + l + Mr }}  \big) \ . 
\end{equation}
We compute, using \eqref{wbound}, \eqref{a0constraint} and \eqref{epsilonconstraint},
\begin{equation}
\begin{split}
T_0
& \geq \frac{r}{|w|}  \big(  1- \sqrt{\frac{l + Mr}{ r^2 w^2 + l + Mr }}  \big)  \geq \frac{r}{|w|} - \frac{\sqrt{l + Mr}}{w^2}  \\
& \geq a_0 \epsilon^2 - \epsilon^3  - \frac{\sqrt{1+ 8 a_0^{-3} b^3 \frac{3}{2} b }}{(\frac{\epsilon}{a_0} - \frac{b}{2 a_0 \epsilon^2} )^2}   \geq a_0 \epsilon^2 - \epsilon^3 - \frac{16 a_0^2 \epsilon^4 }{ ( 2 a_0 \epsilon^3 - b  )^2 }  \\
& >   a_0 \epsilon^2  - \epsilon^3 - \frac{100 a_0^2}{b^2} \epsilon^4   = T  \ . \\
\end{split}
\end{equation}
Therefore, $T \in (0, T_0] $ for every $(r, w, l) \in S_+$. We compute the upper bound for $R(T)$ using \eqref{rbound}, \eqref{wbound} and \eqref{lbound}:
\begin{equation}
\begin{split}
R(T)^2 
& \leq (r + wT)^2 + (l r^{-2} + Mr^{-1} )T^2 \\
& \leq \big[ r + w \big( a_0 \epsilon^2 - \epsilon^3 - \frac{100 a_0^2}{b^2} \epsilon^4  \big) \big]^2 +  (l r^{-2} + Mr^{-1} )\big( a_0 \epsilon^2 - \epsilon^3 - \frac{100 a_0^2}{b^2} \epsilon^4  \big)^2 \\
& \leq (r + a_0 w \epsilon^2 )^2 + \big( w \epsilon^3 + w \frac{100 a_0^2}{b^2} \epsilon^4   \big)^2  + 2 | r +  a_0 w \epsilon^2  | \cdot |w| \big(\epsilon^3 + \frac{100 a_0^2}{b^2} \epsilon^4  \big) \\
& \quad + \big[ 1 \cdot (\frac{1}{2} b)^{-2} + 8 a_0^{-3} b^3 \cdot ( \frac{1}{2} b )^{-1}   \big]  (a_0 \epsilon^2)^2  \\
& \leq (\epsilon^3)^2 + \frac{9}{4} \epsilon^{-4} (\epsilon^3 + \frac{100 a_0^2}{b^2} \epsilon^4 )^2 + 2 \epsilon^3 \cdot \frac{3}{2} \epsilon^{-2} \cdot ( \epsilon^3 + \frac{100 a_0^2}{b^2} \epsilon^4 ) \\
& \quad + ( 4b^{-2} + 16 a_0^{-3} b^2 ) a_0^2 \epsilon^4   \\
& \leq \epsilon^6 + 3 \epsilon^{-4} (2 \epsilon^3)^2 + 3 \epsilon \cdot 2 \epsilon^3 + (4 b^{-2} + 16 a_0^{-3} b^2  ) a_0^2 \epsilon^4  \\
& \leq 16 \epsilon^2   \ . \\
\end{split} 
\end{equation}
The last two lines comes from our choice of the parameters $a_0$ and $\epsilon$. Hence  
\begin{equation}
\sup_{(r, w, l) \in S_+} R(T, r, w, l)  \leq 4 \epsilon  \leq \epsilon_0 \ .
\end{equation}

The lower bounds $ \| \rho (T) \|_{L^\infty_{|x|\leq \epsilon_0}} \  , \ \| E (T) \|_{L^\infty} \geq N $ then follow from Lemma \ref{concentrationlemma} with $K = 4 \epsilon$, together with \eqref{epsilonconstraint}: 
\begin{equation}
\| \rho (T) \|_{L^\infty_{|x|\leq \epsilon_0}} \geq \frac{3}{4 \pi}  \frac{M}{ K^3} \geq \frac{3}{4 \pi}  \frac{\frac{3}{8} a_0^{-3} b^3   }{ 64 \epsilon^3 } \geq N  \ ,  
\end{equation}
\begin{equation}
\| E (T) \|_{L^\infty} \geq  \frac{M}{ K^2} \geq \frac{\frac{3}{8} a_0^{-3} b^3   }{ 16 \epsilon^2 }   \geq N \ .
\end{equation}

This completes the proof of Theorem \ref{mainresultVP}.

\end{proof}

Next we prove Remark 1.2 \textit{ii)}. 

\begin{proof}
We assume $1 \leq k \leq 5$. Recall that we want to take $a_0$ and $\epsilon$ such that \eqref{a0constraint} and \eqref{epsilonconstraint} are satisfied. For $b \in (0, 1)$, from the definition of $d_k$, there exists $\tilde{n}$ independent of $b$, such that 
\begin{equation} \label{tildenbsmall}
\tilde{n} b^{-k} \geq 20 d_{k } \eta^{-1/3}  C_0^{1/3}  \ .  
\end{equation}
Hence we can take $a_0 = \tilde{n} b^{-k}  $ so that \eqref{a0constraint} holds. Also, recall \eqref{a0constraint} and the definition \eqref{epsilonconstraint} of $\beta_0$, we observe that for $b \in (0, 1)$, there exists $\tilde{C}$ independent of $b$, such that 
\begin{equation} \label{tildeCbsmall}
\tilde{C} b^3 \leq \beta_0 \ .
\end{equation}
Hence we can take $\epsilon = \tilde{C} b^3$. With these chosen values of $a_0$ and $\epsilon$, we have 
\begin{equation}  \label{Tdefinitionbleq1}
T = T (b) = a_0 \epsilon^2  - \epsilon^3  - \frac{100 a_0^2}{b^2} \epsilon^4 = \tilde{n} \tilde{C}^2 b^{6-k} - \tilde{C}^3 b^9 - 100 \tilde{C}^4 \tilde{n}^2 b^{10-2k} 
\end{equation} 
for $b \in (0, 1)$.

When $1 \leq k \leq 5$, we first pick $\tilde{n}$ large, then pick $\tilde{C}$ small, such that when $b \in (0, 1)$, the inequalities \eqref{tildeCbsmall} and \eqref{tildenbsmall} as well as $ \tilde{n} \tilde{C}^2 > 9 \tilde{C}^3 + 1000 \tilde{n}^2 \tilde{C}^4 $ are satisfied. With this we can verify that when $b \in (0, 1)$, $T'(b) \geq 0$, and the term $\tilde{n} \tilde{C}^2 b^{6-k} $ is dominant in the expression of $T(b)$. Hence $T(b)$ is an increasing function of $b \in (0,1)$.

Moreover, we construct $T(b)$ for $b >1$ as described in the proof of Theorem \ref{mainresultVP}, so $T(b)$ is an increasing function of $b  \in (1, +\infty)$. Recall that the constraints on $n$ and $C$ are \eqref{ndefinition} and \eqref{Cdefinition}. We take $n$ large enough such that not only \eqref{ndefinition} holds, but also
$$ n C^2   - C^3   - 100 C^4 n^2  \geq \tilde{n} \tilde{C}^2   - \tilde{C}^3  - 100 \tilde{C}^4 \tilde{n}^2  \ . $$
Then \eqref{Tdefinitionbgeq1} gives a function $T= T(b)$ such that $T(b) \geq T(1)$ for all $b>1$. Combining this together with \eqref{Tdefinitionbleq1}, we obtain a function $T(b)$ that is increasing on $(0, +\infty)$ for $1 \leq k \leq 5$. This completes the proof of Remark 1.2 \textit{ii)}. 

\end{proof}


\begin{flushleft}
\textbf{Remark 3.1} \it{The proof of Theorem \ref{mainresultVP} can actually be applied to a system given by the equations \eqref{VPreformulated}, \eqref{mtr}, \eqref{rhotr}, \eqref{electricfieldexpression} on a bounded domain. This gives the following result: \\
Consider the system given by the equations \eqref{VPreformulated}, \eqref{mtr}, \eqref{rhotr}, \eqref{electricfieldexpression} on a bounded domain $\Omega \subset \mathbb{R}^3$ which contains a ball centered at the origin. Then no matter what boundary conditions we put, for any positive integer $k$ and positive constants $\eta$, $N$, as well as any positive constant $b$ satisfying $ b < \frac{2}{3} \inf_{x \in \partial \Omega} r(x) $, there exists some $T> 0$ and a smooth, spherically symmetric solution which has lifespan at least $T$, such that 
\begin{equation}
\supp_x f(0, x, v) = \{  \frac{1}{2} b \leq |x| \leq \frac{3}{2} b   \} 
\end{equation}
and 
\begin{equation}  
\| \rho (0) \|_{C^k} \  , \ \| E (0) \|_{C^k} \leq \eta \ , 
\end{equation}
while 
\begin{equation} 
\| \rho (T) \|_{L^\infty_{|x|\leq \epsilon_0}} \  , \ \| E (T) \|_{L^\infty } \geq N \ . 
\end{equation}
}
\end{flushleft}


\section{Focusing Solutions of the Relativistic VP system}

In this section, we are going to prove, for the system RVP:
\begin{theorem} \label{mainresultRVP}
For any positive integer $k$ and positive constants $\eta$, $N$, there exists a smooth, spherically symmetric solution of RVP, such that 
\begin{equation}  
\| \rho (0) \|_{C^k} \  , \ \| E (0) \|_{C^k} \leq \eta \ , 
\end{equation}
while for some $T >0$ and some $\epsilon_0 >0$, 
\begin{equation} 
\| \rho (T) \|_{L^\infty_{|x|\leq \epsilon_0}} \  , \ \| E (T) \|_{L^\infty } \geq N \ , 
\end{equation}
\end{theorem}

\begin{proof}
Without the loss of generality, we assume $0 <\eta< 1$ and $N >1$ in the proof. 

We take 
\begin{equation} \label{ep}
\epsilon = \min \{\frac{1}{1000}, \big( \frac{\epsilon_0}{110} \big)^{1/3} , \frac{1}{32} \eta ,  \frac{1}{4} \big(\eta C_0^{-1} (1+c_k)^{-1} \big)^{8/15} ,  10^{-16} N^{-4} \} \ .  
\end{equation}
where $C_0$ is a constant which will be explained in the proof. Set 
\begin{equation} \label{constants}
A = \epsilon^5  , \ B= \epsilon^{-3/2} , \ b = \epsilon^{-7/8} \ ,
\end{equation}
and take
\begin{equation}
T = b - 10 \epsilon^{35/8} > 0 \ .
\end{equation}

Let $H : [0, \infty) \rightarrow [0, \infty)$ be any function satisfying $ \int_{\mathbb{R}^3} H(|u|^2) du = \frac{3}{4 \pi} $ with $\supp (H) \subset [0, 1]$, and rescale it for any $\epsilon \in (0, 1)$:
$$ H_\epsilon (|u|^2) = \frac{1}{\epsilon^9} H \big( \frac{|u|^2}{\epsilon^6} \big)   \ . $$
so that $ \int_{\mathbb{R}^3} H_\epsilon(|u|^2) du = \frac{3}{4 \pi} $ and $\supp (H_\epsilon) \subset [0, \epsilon^6]$. For any $\epsilon >0$, $x$, $v \in \mathbb{R}^3$, define 
$$ h_{\epsilon } (x, v) = H_\epsilon \big(  | \frac{x}{A} + B v |^2 \big) \ . $$
Moreover, we choose the cut-off function $\chi \in C^\infty ((0, \infty); [0, 1] ) $ satisfying $\chi (r) = 0  $ for $r <   b- \epsilon^3 $ or $r >  b+ \epsilon^3 $, $\chi (r) =1$ for $r \in [ b - \frac{1}{2} \epsilon^3, b + \frac{1}{2} \epsilon^3 ]$. For each $k$, there exists $c_k >0$ independent of $\epsilon$ (but depends on $k$), such that $ \| \chi (r) \|_{C^k}  \leq  c_k \epsilon^{-3}$.

We choose the initial data to be 
\begin{equation}
f_0 (x, v) = h_{\epsilon} (x, v) \chi (|x|) \Gamma (\frac{x \cdot v}{r}) \ . 
\end{equation}
where $\Gamma (w) = 0$ for $w \geq 0$, and $\Gamma (w) =1$ for $w <0$. Hence $S(0) \setminus S_+ $ is a zero-measure set.  

For any $(r, w, l) \in S(0)$, we have $ |\frac{x}{A} + B v  |^2  < \epsilon^6 $. Using the angular coordinates and the identity $ |v|^2 = w^2 + lr^{-2} $, this inequality becomes
\begin{equation}
\big( \frac{r}{A} + B w \big)^2 + l \big( \frac{B}{r} \big)^2 < \epsilon^6  \ .
\end{equation}  
The inequality $ | \frac{r}{A} + B w |^2 < \epsilon^6 $ together with \eqref{constants} and \eqref{ep} implies that for $(r, w, l) \in S(0)$,
\begin{equation}
\frac{1}{2} \epsilon^{-7/8}  < b - \epsilon^3 < r < b + \epsilon^3 < \frac{3}{2} \epsilon^{-7/8} \ , 
\end{equation}
\begin{equation}
- 2\epsilon^{-35/8} < - \epsilon^{9/2} - \frac{r}{\epsilon^{7/2}} < w <  - \frac{r}{\epsilon^{7/2}} +  \epsilon^{9/2}  < - \frac{1}{2} \epsilon^{-35/8} \ , 
\end{equation}
\begin{equation}
l < \big( \frac{r}{B} \big)^2 \epsilon^6 <1 \ .
\end{equation}
Notice that the cut-off $\Gamma (w)$ does not affect the smoothness of $f_0$ and $\rho (0)$, since $w$ is bounded away from $0$.

We have 
$$  \rho (0, r) = \int_{\mathbb{R}^3} h_{\epsilon } (x, v) \chi (|x|) \Gamma (\frac{x \cdot v}{r}) dv = \frac{3}{8 \pi B^3} \chi (r) \ .   $$

We compute
\begin{equation} \label{MsizeRVP1}
\begin{split}
M = \int_{\mathbb{R}^3} \rho_0 (x) dx 
& = 4 \pi \int^{b+\epsilon^3 }_{b-\epsilon^3} \rho_0 (r) r^2 dr \\
& \leq \frac{1}{2} B^{-3} \big[  (b+ \epsilon^3)^3 - (b-\epsilon^3)^3  \big]  \\
& = \frac{1}{2} B^{-3}  ( 6 b^2 \epsilon^3 + 2 \epsilon^9  ) \\
& \leq 4 B^{-3}  \epsilon^3 b^2 \ , \\
\end{split}
\end{equation}
and
\begin{equation} \label{MsizeRVP2}
\begin{split}
M = \int_{\mathbb{R}^3} \rho_0 (x) dx 
& \geq 4 \pi \int^{b+ \frac{1}{2} \epsilon^3 }_{b- \frac{1}{2} \epsilon^3} \rho_0 (r) r^2 dr \\
& \geq \frac{1}{2} B^{-3}  \big[  (b+ \frac{1}{2} \epsilon^3)^3 - (b- \frac{1}{2} \epsilon^3)^3  \big]  \\
& =  \frac{1}{2} B^{-3}  ( 3 b^2 \epsilon^3 + \frac{1}{4} \epsilon^9  ) \\
& \geq \frac{3}{2} B^{-3} \epsilon^3 b^2 \ .  \\
\end{split}
\end{equation}

Using \eqref{mtr}, \eqref{rhotr}, \eqref{electricfieldexpression}, we obtain 
\begin{equation} \label{rho0boundRVP}
 \| \rho (0) \|_{L^\infty} \leq \frac{3}{8 \pi B^3}  \leq \eta \ , 
\end{equation} 
and 
\begin{equation} \label{E0boundRVP}
 \| E (0) \|_{L^\infty} \leq \frac{4 M}{b^2} \leq  \frac{ 16 B^{-3} \epsilon^3 b^2}{b^2} \leq  \eta  \ .
\end{equation}
Again by using \eqref{mtr}, \eqref{rhotr}, \eqref{electricfieldexpression}, we have
\begin{equation}
\| \rho (0) \|_{C^k_x}  \leq \sum_k  \| \frac{\partial^k}{\partial r^k}  ( \frac{3}{8 \pi a_0^3} \chi (r)) \|_{L^\infty}   \leq  k  \frac{3 B^{-3} }{8 \pi  }  c_k  \leq \eta \ ,
\end{equation} 
since $ B \geq  c_k^{1/3}  \eta^{-1/3} k^{1/3} $. 

Next we prove
$$ \| E (0) \|_{C^k_x}   \leq  \eta  \ .  $$
For any $k$, the $C^k$ norm of $E(0)$ is controlled by $ 2 C_0 (1 + c_k) (1+ b^3) B^{-3}  $, where $C_0$ is a constant only depending on $k$ and $b$. In case $k=1$, we compute, using \eqref{MsizeRVP1}, \eqref{MsizeRVP2}, \eqref{rho0boundRVP} and \eqref{E0boundRVP},
\begin{equation}
\begin{split}
| \partial_{x_j} E_i (0, r) |
& \leq  | \partial_{x_j}  ( \frac{m (0, r)}{r^3} x_i )   |  \\
& \leq  |  \frac{\partial_{x_j} m (0, r) }{r^3} x_i  | + | m (0, r) \frac{-3 x_j x_i }{r^5} | + | \frac{m(0, r) }{r^3} \delta_{ij} |  \\
& \leq 4 \pi \big[  | r^2 \rho (0, r) \frac{x_j}{r} \frac{x_i }{r^3} x_i  | + M |   \frac{-3 x_j x_i }{r^5} | + M | \frac{1}{r^3} \delta_{ij} |  \big] \\
& \leq   4 \pi \cdot 5 \cdot (1+c_k) \max \{ M, \frac{3}{ 4 \pi B^{3} } \} \max \{ 1, \frac{1}{(\frac{1}{2} b)^3} \}    \\
& \leq   1500 \pi (1+ b^3) (1+ b^{-3})  (1+ c_k) B^{-3}   \ , \\
\end{split}
\end{equation} 
so we can take $C_0  = 32  + 1500 \pi   $. Similarly, in case $k=2$, we have
\begin{equation}
\begin{split}
& \quad | \partial^2_{x_j x_l} E_i (0, r) | \\
& \leq  | \partial_{x_l}  (  \frac{\partial_{x_j} m (0, r) }{r^3} x_i  +  m (0, r) \frac{-3 x_j x_i }{r^5}  + \frac{m(0, r) }{r^3} \delta_{ij})   |  \\
& \leq  |  \frac{\partial_{x_j x_l} m (0, r) }{r^3} x_i  | + | \partial_{x_j} m (0, r) \frac{-3 x_j x_i }{r^5} | + | \frac{\partial_{x_j} m(0, r) }{r^3} \delta_{il} |  \\
& \quad + |  \frac{\partial_{x_l} m (0, r) }{r^3} x_i x_j  (-3)| + | m (0, r) \frac{(-3) (-5) x_l x_i x_j }{r^7} | + | \frac{m(0, r) }{r^5} (-3) (x_i \delta_{jl} + x_j \delta_{il} ) |  \\
& \quad + |  \frac{\partial_{ x_l} m (0, r) }{r^3} \delta_{ij}  | + | m (0, r) \frac{-3 x_l  }{r^5} \delta_{ij} | + | \frac{m(0, r) }{r^3} \cdot 0 |  \\
& \leq 4 \pi \big[ |  \frac{\partial_{x_j }  r^2  \rho (0, r) \frac{x_l}{r} }{r^3} x_i  | + | r^2 \rho (0, r) \frac{x_j}{r} \frac{-3 x_j x_i }{r^5} | + | r^2 \rho (0, r) \frac{ x_j }{r} \frac{1}{r^3} \delta_{il} |  \\
& \quad + | r^2 \rho(0, r) \frac{x_l}{r} \frac{-3 x_i x_j}{r^5}  | + M |   \frac{(-3) (-5) x_l x_i x_j }{r^7} | + M | \frac{1 }{r^5} (-3) (x_i \delta_{jl} + x_j \delta_{il} ) |  \\
& \quad + | r^2 \rho (0, r) \frac{x_l}{r}  \frac{1}{r^3} \delta_{ij}  | +  M |  \frac{-3 x_l  }{r^5} \delta_{ij} |   \big] \\
& \leq   12000 \pi (1+ b^3) (1+ b^{-4})  (1+ c_k) B^{-3}   \ , \\
\end{split}
\end{equation} 
so we can take $C_0  = 32  + 13500 \pi $ for $k=2$.   

For other $k$'s $C_0$ can be deduced similarly. Since $ \epsilon \leq \frac{1}{4} \big(c_1 C_0^{-1} (1+c_k)^{-1} \big)^{8/15} $, we have
$$ \| E (0) \|_{C^k_x}  \leq 2 C_0 (1 + c_k) (1+ b^3) B^{-3}  \leq  \eta  \ .  $$

We choose $T_0 = T_0 (r, w, l) > 0$ for $(r, w, l) \in S_+$ as in Lemma \ref{behaviorofcharacteristics}. Therefore
\begin{equation}
\frac{d R}{dt} \leq 0 \ , \ t \in [0, T_0] \ , 
\end{equation}
\begin{equation}
D \leq 1 + 8 B^{-3} \epsilon^3 b^2 \cdot 2 b \sqrt{1 + 4 \epsilon^{-35/4} + 4 \epsilon^{7/4} }  \leq  4 \ . 
\end{equation}
\begin{equation}
T_0 > r - \frac{\sqrt{D}}{w} \geq b - 8 \epsilon^{35/8} > T = b - 10 \epsilon^{35/8}  \ .
\end{equation}

Moreover, from Lemma \ref{behaviorofcharacteristics} we have
\begin{equation}
R(T)^2 \leq \big( r - \frac{|w|}{ \sqrt{1+ w^2 + l r^{-2}}} T  \big)^2 + \frac{D}{r^2 (1+ w^2 + l r^{-2} ) } T^2 : = I + II \ .
\end{equation}
We define $g(x) =\frac{1- x}{1+2x} $, so $g'(0) = -3$, $g''(x) >0$. We compute 
\begin{equation}
\begin{split}
\frac{|w|}{ \sqrt{1+ w^2 + l r^{-2}}} 
& \geq  \frac{ r \epsilon^{-7/2} - \epsilon^{9/2} }{ \sqrt{  1+ ( r \epsilon^{-7/2} + \epsilon^{9/2} )^2  + 4 \epsilon^{7/4}  } }   \\ 
& \geq  \frac{1/2- \epsilon^{71/8} }{ \sqrt{   ( 3/2 + \epsilon^{71/8} )^2  + 2 \epsilon^{35/4}  } }   \\ 
& \geq  \frac{1- 2 \epsilon^{35/4} }{ \sqrt{   ( 1 + 4 \epsilon^{35/4}  )^2 } }   \\ 
& = g(2 \epsilon^{35/4} ) \\
& \geq g(0) + g'(0) \cdot 2 \epsilon^{35/4}  \\
& = 1 - 6 \epsilon^{35/4}  \ . \\
\end{split}
\end{equation}

It follows that
\begin{equation}
\begin{split}
I 
& \leq \big[  r - (1 - 6 \epsilon^{35/4} ) T \big]^2 \\
& \leq \big[  b + \epsilon^3 - (1 - 6 \epsilon^{35/4} ) (b - 10 \epsilon^{35/8} ) \big]^2 \\
& \leq 10 \big( \epsilon^6 +36 b^2 \epsilon^{35/2} + 100 \epsilon^{35/4} + 3600 \epsilon^{35/2} \epsilon^{35/4} \big) \\
& \leq 10 \big( \epsilon^6 +36 \epsilon^{63/4 } + 100 \epsilon^{35/4} + 3600 \epsilon^{105/4}  \big) \\
& = 10000 \epsilon^6 \ . \\
\end{split}
\end{equation}
Moreover, 
\begin{equation}
II \leq \frac{D}{r^2 w^2} T^2 \leq \frac{4}{\frac{1}{4} b^2 \cdot \frac{1}{4} \epsilon^{-35/4} } b^2 \leq 256 \epsilon^{35/4} \ .
\end{equation}  
We conclude $R(T) \leq 110 \epsilon^3 $. We have
\begin{equation}
\sup_{(r, w, l) \in S(0)} R(T, r, w, l) = \sup_{(r, w, l) \in S_+} R(T, r, w, l)  \leq 110 \epsilon^3   \ .
\end{equation}
 
Applying Lemma \ref{concentrationlemma} with $K = 110 \epsilon^3 \leq \epsilon_0 $, we obtain, due to the choice of $\epsilon$:
\begin{equation}
\| \rho (T) \|_{L^\infty_{|x|\leq \epsilon_0}} \geq \frac{3}{4 \pi}  \frac{M}{ K^3} \geq \frac{3}{4 \pi}  \frac{3 B^{-3} \epsilon^3 b^2  }{ 1331000 \epsilon^9 } \geq N  \ ,  
\end{equation}
\begin{equation}
\| E (T) \|_{L^\infty} \geq  \frac{M}{ K^2} \geq \frac{3 B^{-3} \epsilon^3 b^2  }{ 12100 \epsilon^6 }   \geq N \ .
\end{equation}
This completes the proof of the theorem.

\end{proof}

\section{Acknowledgement}

The author thanks her advisor, Walter Strauss for all the guidance, encouragement and patience. Also, she thanks Jonathan Ben-Artzi for helpful discussions.

\end{document}